\documentclass[11pt]{article}
\usepackage{url}
\usepackage{amsfonts}
\usepackage{amssymb}
\usepackage{amsmath}
\usepackage{amsthm}
\usepackage[lofdepth,lotdepth]{subfig}
\usepackage{setspace,amssymb,verbatim}
\usepackage[margin=1in]{geometry}
\usepackage{graphicx}
\usepackage{rotating}
\usepackage{psfrag}
\usepackage{afterpage}
\usepackage{multirow}
\usepackage[lofdepth,lotdepth]{subfig}
\usepackage{tikz}

\theoremstyle{plain}
\newtheorem{theorem}{Theorem}[section]
\newtheorem{proposition}[theorem]{Proposition}
\newtheorem*{theorem*}{Theorem}
\newtheorem{corollary}[theorem]{Corollary}

\newtheorem*{lemma*}{Lemma}
\newtheorem{lemma}[theorem]{Lemma}

\numberwithin{equation}{section}

\theoremstyle{remark}

\newtheorem*{case*}{Case}

\theoremstyle{definition}

\usepackage{prettyref}
\newrefformat{sec}{Section \ref{#1}}
\newrefformat{chap}{Chapter \ref{#1}}
\newrefformat{prop}{Proposition \ref{#1}}
\newrefformat{prob}{Problem \ref{#1}}
\newrefformat{rem}{Remark \ref{#1}}
\newrefformat{cor}{Corollary \ref{#1}}
\newrefformat{cond}{Condition \ref{#1}}
\newrefformat{def}{Definition \ref{#1}}
\newrefformat{ex}{Example \ref{#1}}
\newrefformat{lem}{Lemma \ref{#1}}
\newrefformat{eq}{Equation \eqref{#1}}
\newrefformat{con}{condition (\ref{#1})}
\newrefformat{conj}{Conjecture \ref{#1}}
\newrefformat{fig}{Figure \ref{#1}}
\newrefformat{tab}{Table \ref{#1}}
\newrefformat{app}{Appendix \ref{#1}}

\newcommand{\F}{\mathbb{F}}

\newcommand{\I}{\mathcal{I}}
\newcommand{\U}{\mathcal{U}}
\newcommand{\cP}{\mathcal{P}}
\newcommand{\cC}{\mathcal{C}}
\newcommand{\bmu}{\boldsymbol{\mu}}

\def\adots{\mathinner{\mkern2mu\raise0pt\hbox{.}  
\mkern2mu\raise4pt\hbox{.}\mkern1mu
\raise7pt\vbox{\kern7pt\hbox{.}}\mkern1mu}}

\title{Real classes of finite special unitary groups}

\author{Amanda A. Schaeffer Fry\\ \small\textit{Department of Mathematical and Computer Sciences}\\\small\textit{Metropolitan State University of Denver}\\ \small\textit{Denver, CO 80217, USA}\\ \small\textit{aschaef6@msudenver.edu}
\and
C. Ryan Vinroot\\ \small\textit{Department of Mathematics}\\\small\textit{College of William and Mary}\\\small\textit{Williamsburg, VA 23187, USA}\\ \small\textit{vinroot@math.wm.edu}}
\date{}

\begin{document}

\maketitle

\begin{abstract} 
We classify all real and strongly real classes of the finite special unitary group $SU_n(q)$.  Unless $q \equiv 3($mod $4)$ and $n|4$, the classification of real classes is similar to that of the finite special linear group $SL_n(q)$.  We relate strong reality in $SU_n(q)$ to strong reality in the finite special orthogonal groups $SO^{\pm}_n(q)$, and we classify real and strongly real classes in the last case for the group $SO_n^{\pm}(q)$, when $q$ is odd and $n \equiv 2($mod $4)$.
\end{abstract}

\section{Introduction}

An element $g$ in a group $G$ is \emph{real} in $G$ if there is an element $h \in G$ such that $hgh^{-1} = g^{-1}$, and such an element $h$ will be called a \emph{reversing element}.  If the reversing element $h$ further satisfies $h^2 = 1$, then we say that the element $g$ is \emph{strongly real} in $G$.  An element $g \in G$ is real or strongly real if and only if each element in its $G$-conjugacy class is as well, and so we may speak of real and strongly real classes of $G$.

In the case that $G$ is a finite group, the number of real classes of $G$ is equal to the number of complex irreducible representations of $G$ which have real-valued characters.  Brauer \cite[Problem 19]{Br63} asked for a group-theoretical characterization for the number of complex irreducible representations of $G$ which can be realized over the field of real numbers.  While there is still no satisfactory solution to this problem, there are many examples which point to the notion of strong reality being an important concept in an eventual solution, see \cite{Go79} for example.  This apparent connection has spurred the classification of real and strongly real classes in various families of finite groups.

For example, Vdovin and Gal't \cite{VdGa10} finished the final cases in the proof that if $G$ is a finite simple group such that all of the elements of $G$ are real, then all of the elements of $G$ are strongly real.  Tiep and Zalesski \cite{TiZa05} have classified all finite simple and quasi-simple groups with the property that all of their classes are real.  Gill and Singh \cite{GiSi11, GiSi112} have classified all of the real and strongly real classes of the finite special and projective linear groups, as well as all quasi-simple covers of the finite projective special linear groups.  Gates, Singh, and the second-named author of this paper \cite{GSV14} classified the strongly real classes of the finite unitary group.

In this paper, we classify all of the real and strongly real classes of the finite special unitary groups.  Apart from continuing the program already in place for such classifications, we are also motivated by a result of Gow \cite{Go84} which gives a natural bijection between the real classes of the finite general linear groups and the finite unitary groups.  Since the classification of real classes of the finite special linear groups has been completed, it is a natural question to ask what happens with the bijection described by Gow when considering the real classes of the finite special linear and special unitary groups.  As we see, in most cases the description of real classes of the finite special unitary groups $SU_n(q)$ is essentially the same as the finite special linear groups $SL_n(q)$, but in the case that $q \equiv 3($mod $4)$ and $4|n$, the description is quite different.

This paper is organized as follows.  In Section \ref{Prelim}, notation and previous results on real and strongly real classes for the relevant groups are given.  In Section \ref{JdecompCent}, we discuss the Jordan decomposition of elements, and the description of centralizers of unipotent elements, which is key to our main arguments.  We prove some initial results in Section \ref{initialobs}, many of which are direct adaptations from \cite{GiSi11}.   In Section \ref{RealUnis}, we classify the real unipotent classes in the finite special unitary groups, and in Section \ref{MainThm} we prove our main result in Theorem \ref{MainReal} where we classify all real elements in these groups.  Finally, in Section \ref{strongreal}, we classify the strongly real elements in the finite special unitary groups in Theorem \ref{SUstrong}, and relate this classification to strong reality in the special orthogonal groups in Corollary \ref{SUSOstrong}.\\
\\
\noindent{\bf Acknowledgements.  }  The authors thank Bob Guralnick, Nick Gill, and Anupam Singh for helpful conversation and communication.  Schaeffer Fry and Vinroot were each supported in part by a grant from the Simons Foundation (Awards \#351233 and \#280496, respectively).

\section{Preliminaries and notation} \label{Prelim}

For any positive integer $n$, we let $(n)_2$ denote the $2$-part of $n$, that is, the largest power of $2$ which divides $n$, and we let $(n)_{2'} = n/(n)_2$ denote the odd part of $n$, or the largest odd divisor of $n$.

\subsection{The groups} \label{groups}

Let $\F_q$ be a finite field with $q$ elements of characteristic $p$, and let $\bar{\F}_q$ be a fixed algebraic closure.  Let $GL_n$ denote the group of $n$-by-$n$ invertible matrices over $\bar{\F}_q$, and let $F$ denote the standard Frobenius morphism of $GL_n$ defined by $F((a_{ij})) = (a_{ij}^q)$.  We denote the group of $F$-fixed points of $GL_n$, the finite general linear group, by
$$GL_n(q) = GL_n^F.$$
The finite special linear group $SL_n(q)$ is then defined as the set of determinant $1$ elements of $GL_n(q)$, 
$$SL_n(q) = \{ g \in GL_n(q) \, \mid \, \mathrm{det}(g) = 1 \}.$$

Now let $\eta$ be some invertible $n$-by-$n$ Hermitian matrix over $\F_{q^2}$, so that if $\eta = (\eta_{ij})$ then ${^\top \eta} = (\eta_{ji}) = (\eta_{ij}^q)$.  Define another Frobenius morphism $F_{\eta}$ of $GL_n$ (dependent on $\eta$) by
$$F_{\eta}(g) = \eta^{-1} ({^\top F}(g)^{-1}) \eta,$$
which is the standard Frobenius twisted by a graph automorphism of $GL_n$.  Define the unitary group $GU_n(q)$ to be
$$GU_n(q) = GL_n^{F_{\eta}},$$
and the special unitary group $SU_n(q)$ to be the set of determinant $1$ elements of $GU_n(q)$,
$$ SU_n(q) = \{ g \in GU_n(q) \, \mid \, \mathrm{det}(g) = 1 \}.$$

It will often be convenient to vary the choice of the Hermitian matrix $\eta$.  While this choice changes the matrices which are elements of $GU_n(q)$ or $SU_n(q)$, any two groups obtained by varying $\eta$ are isomorphic, and are in fact conjugate subgroups of $GL_n$.  It will also be useful to consider the fact that we have natural embeddings $GU_n(q) \subset GL_n(q^2)$ and $SU_n(q) \subset SL_n(q^2)$. 

We identify $GL_1$ with $\bar{\F}_q^{\times}$, so $GL_1(q)$ is identified with $\F_q^{\times}$, and $GU_1(q)$ is identified with a cyclic subgroup of order $q+1$ inside of $\F_{q^2}^{\times}$.   We let $C_{q+1}$ denote this cyclic subgroup of order $q+1$.

For odd $q$ we let $O^{\pm}_{n}(q)$ denote the split ($+$) or non-split ($-$) orthogonal group over the field $\F_q$.  That is, $O^{\pm}_n(q)$ is the subgroup of elements of $GL_n(q)$ which stabilize a non-degenerate (split or non-split) symmetric form on $\F_q^n$, and when $n$ is odd we have $O_n(q) = O^+_n(q) = O^-_n(q)$.  We let $SO^{\pm}_n(q)$ denote the special orthogonal group, or the subgroup of determinant $1$ elements of $O^{\pm}_n(q)$.

\subsection{Conjugacy classes} \label{classes}
We let $\cP$ denote the set of all integer partitions.  If $\mu \in \cP$, then we either write the parts of $\mu$ as $\mu = (\mu_1, \mu_2, \ldots)$ with $\mu_i \geq \mu_{i+1}$, or $\mu = (1^{m_1} 2^{m_2} \cdots )$, where $m_k$ is the multiplicity of the part $k$ in $\mu$, and we also write $m_k = m_k(\mu)$.  We write $|\mu|$ as the size of $\mu$, so $|\mu| = \sum_i \mu_i = \sum_k k m_k(\mu)$.  The only partition of size zero is written as $\varnothing$, the empty partition.

The conjugacy classes of $GL_n(q)$ are determined by elementary divisors.  In particular, if $\I$ is the set of non-constant monic irreducible polynomials in $\F_q[t]$ with non-zero constant, then conjugacy classes of $GL_n(q)$ are in bijection with functions $\bmu: \I \longrightarrow \cP$ which satisfy $\sum_{f \in \I} \mathrm{deg}(f) |\bmu(f)| = n$, where $|\bmu(f)|$ is the size of the partition $\bmu(f)$.  If the partition $\bmu(f)$ has parts $\bmu(f) = (\bmu(f)_1, \bmu(f)_2, \ldots)$, then the conjugacy class of $GL_n(q)$ corresponding to $\bmu$ has elementary divisors $f^{\bmu(f)_1}, f^{\bmu(f)_2}, \ldots$.
The unipotent classes of $GL_n(q)$ are thus given by functions $\bmu$ such that $\bmu(f) = \varnothing$ unless $f(t) = t-1$.  In this way, we may parameterize unipotent classes of $GL_n(q)$ by partitions $\mu$ of $n$.  The semisimple classes of $GL_n(q)$ then correspond to those functions $\bmu$ such that for each $f \in \I$, $\bmu(f)$ has parts that are at most size $1$, so $\bmu(f) = (1^{|\bmu(f)|})$.  This is equivalent to an element in the conjugacy class over $\bar{\F}_q$ having all Jordan blocks of size $1$.

Note that for any irreducible polynomial $f \in \I$, there is an element $\alpha \in \bar{\F}_q^{\times}$ such that
$$ f(t) = (t - \alpha) (t-\alpha^q) \cdots (t-\alpha^{q^{d-1}}),$$
where $d = \mathrm{deg}(f)$.  In this way, the elements of $\I$ are in bijection with the $F$-orbits of the elements of $\bar{\F}_q^{\times}$.  We may also consider the $F_{\eta}$-orbits of $\bar{\F}_q^{\times}$, where $F_{\eta}(\alpha) = \alpha^{-q}$.  For each such orbit, we associate a polynomial $f$,
$$f(t) = (t - \alpha) (t-\alpha^{-q})\cdots (t-\alpha^{(-q)^{d-1}}),$$
where $d$ is the size of the $F_{\eta}$-orbit of $\alpha$, and is also the degree of $f$.  We denote this set of polynomials by $\U$.  This set of polynomials takes the place of $\I$ in description of conjugacy classes of $GU_n(q)$, as determined by Ennola \cite{En62} and Wall \cite{Wa62}.  We in fact have $\U \subset \F_{q^2}[t]$, and another way of describing $\U$ is the set of monic non-constant polynomials with nonzero constant in $\F_{q^2}[t]$ which satisfy $f = f^{\checkmark}$, where if $f$ has roots $\alpha_1, \ldots, \alpha_d$ in $\bar{\F}_q^{\times}$, then $f^{\checkmark}$ is the polynomial with roots $\alpha_1^{-q}, \ldots, \alpha_d^{-q}$, and such that $f$ has no proper factors in $\F_{q^2}[t]$ with the same property.  Then $f \in \U$ is also irreducible in $\F_{q^2}[t]$ if and only if $d = \mathrm{deg}(f)$ is odd, and every root $\alpha$ of $f$ satisfies $\alpha^{q^d + 1} = 1$.  Otherwise, $f = g g^{\checkmark}$ where $g$ is irreducible in $\F_{q^2}[t]$.

The conjugacy classes of $GU_n(q)$ may be parameterized by functions $\bmu: \U \longrightarrow \cP$ which satisfy $\sum_{f \in \U} \mathrm{deg}(f) |\bmu(f)| = n$.  In this case, we take $f^{\bmu(f)_1}, f^{\bmu(f)_2}, \ldots$ to be the elementary divisors of elements of $GU_n(q)$ in the class corresponding to $\bmu$.  Like in $GL_n(q)$, the unipotent classes of $GU_n(q)$ may be parameterized by partitions $\mu$ of $n$, and the semisimple classes of $GU_n(q)$ correspond to those $\bmu$ such that $\bmu(f) = (1^{|\bmu(f)|})$ for each $f \in \U$.

If $\bmu$ parameterizes a conjugacy class of $GL_n(q)$ or $GU_n(q)$ as above, we let $\cC_{\bmu}$ denote this class.  Given the class $\cC_{\bmu}$, for each positive integer $i$, define the polynomial $w_i(t)$ by
$$ w_i(t) = \prod_{f} f(t)^{m_i(\bmu(f))},$$
where the product is over either $\I$ or $\U$, in the cases $GL_n(q)$ or $GU_n(q)$, respectively.  That is, $w_i(t)$ is the product, over all $f$, with a factor of $f$ for each time the integer $i$ is the part of the partition $\bmu(f)$.  Then let $n_i = \mathrm{deg}(w_i(t)) = \sum_f \mathrm{deg}(f) m_i(\bmu(f))$.  We define the \emph{type} of the class $\cC_{\bmu}$, following Macdonald \cite{Ma81}, to be the partition $\nu$ of $n$ such that $n_i = m_i(\nu)$, so $\nu = (1^{n_1} 2^{n_2} \cdots )$.

The conjugacy classes of $SL_n(q)$ and $SU_n(q)$ were described (and enumerated) by Macdonald \cite{Ma81} (see also \cite{BrWi11}).  Given a class $\cC_{\bmu}$ of $GL_n(q)$ or $GU_n(q)$, then it follows from the definition that $\cC_{\bmu}$ is contained in $SL_n(q)$ or $SU_n(q)$ exactly when the polynomial $\prod_{f, i} f^{\bmu(f)_i}$ has constant term $(-1)^n$, so that the corresponding elements have determinant $1$.  The question is then how such a class $\cC_{\bmu}$ splits in $GL_n(q)$ or $GU_n(q)$.  Macdonald \cite[Section 3]{Ma81} showed that if $\cC_{\bmu}$ is a $GL_n(q)$-class contained in $SL_n(q)$ which is of type $\nu = (\nu_1, \nu_2, \ldots, \nu_s)$, then $\cC_{\bmu}$ splits into exactly $\gcd(q-1, \nu_1, \ldots, \nu_s)$ classes in $SL_n(q)$.

The result for $SU_n(q)$ is similar, and is implied in \cite[Section 6]{Ma81}.  If $\cC_{\bmu}$ is a $GU_n(q)$-class of type $\nu = (\nu_1, \nu_2, \ldots, \nu_s)$ which is contained in $SU_n(q)$, then $\cC_{\bmu}$ splits into exactly $\gcd(q+1, \nu_1, \ldots, \nu_s)$ classes in $SU_n(q)$.  This follows by applying \cite[(3.2)]{Ma81} in a proof analogous to \cite[(3.3)]{Ma81}, where in \cite[(3,3), Eq. (5)]{Ma81}, in the case $SU_n(q)$, we have $\pi(x,y) = \pi(x, y^{-q})$ with $y^{1-q} \in C_{q+1}$, so that $\F_q^{\times}$ is replaced by the subgroup $C_{q+1}$ of $\F_{q^2}^{\times}$.  The resulting number is also the index of $C_{GU_n(q)}(g) \cdot SU_n(q)$ in $GU_n(q)$ for $g \in \cC_{\bmu}$.

\subsection{Real and strongly real classes} \label{realclasses}

In this section we discuss some of the known results in the classification of real and strongly real classes in finite linear and unitary groups.  

For any monic polynomial $f \in \bar{\F}_q[t]$ with nonzero constant, if $f$ has roots $\alpha_1, \ldots, \alpha_d \in \bar{\F}_q$ with multiplicities, so $f(t) = \prod_{i=1}^d (t - \alpha_i)$, then we define the monic polynomial with nonzero constant $\tilde{f} \in \bar{\F}_q[t]$ by $\tilde{f}(t) = \prod_{i=1}^d (t - \alpha_i^{-1})$.  If $f = \tilde{f}$, we say that $f$ is \emph{self-reciprocal}.

If $f \in \I$ or $\U$, it follows that $\tilde{f} \in \I$ or $\U$, respectively.  If $\cC_{\bmu}$ is a conjugacy class in $GL_n(q)$ or $GU_n(q)$, then $\cC_{\bmu}$ is a real class in that group if and only if $\bmu(f) = \bmu(\tilde{f})$ for every $f \in \I$ or $\U$, respectively.  This follows from considering the elementary divisor theory, and is discussed in \cite{Go84}, for example.  An element $f \in \U$ which is self-reciprocal, other than $t \pm 1$, must be of even degree. Such elements turn out to be elements of $\I$ as well.  If $f \in \U$ is not self-reciprocal, then there is a unique pair $f_1, \tilde{f_1} \in \I$ such that $f \tilde{f} = f_1 \tilde{f_1}$, which follows by considering the corresponding sets of roots as Frobenius orbits.  These facts give a natural bijection between the real classes of $GL_n(q)$ and the real classes of $GU_n(q)$, first described by Gow \cite{Go84}.

It is known that every real class of $GL_n(q)$ is also strongly real, which was first shown for odd $q$ by Wonenburger \cite{Wo66} and in general by by Djokovi\'{c} \cite{Dj67} (these statements are also proved independently in \cite{GiSi11}).  This statement is far from true in the group $GU_n(q)$, as given by the following main result of \cite{GSV14}.

\begin{theorem} [Gates, Singh, and Vinroot] \label{gustrong}  Suppose that $q$ is odd and $g \in GU_n(q)$ is a real element of $GU_n(q)$ in the class $\cC_{\bmu}$.  Then $g$ is strongly real in $GU_n(q)$ if and only if every elementary divisor of the form $(t \pm 1)^{2k}$ of $g$ has even multiplicity.  That is, the real class $\cC_{\bmu}$ of $GU_n(q)$ is strongly real if and only if when $f(t) = t \pm 1$, then $m_{2k}(\bmu(f))$ is always even.  Equivalently, $g \in GU_n(q)$ is strongly real in $GU_n(q)$ if and only if $g$ is an element of some embedded orthogonal group $O^{\pm}_n(q)$.
\end{theorem}

It is a natural question to ask which real or strongly real elements of $GL_n(q)$ or $GU_n(q)$ that are contained in $SL_n(q)$ remain real or strongly real in $SL_n(q)$ or $SU_n(q)$.  Gill and Singh answer the question of which elements of $SL_n(q)$ that are real in $GL_n(q)$ are also real in $SL_n(q)$ in \cite[Propositions 4.4 and 5.5]{GiSi11}, with the following classification. 

\begin{theorem} [Gill and Singh] \label{slnreal} Suppose that $g \in SL_n(q)$ such that $g$ is a real element of $GL_n(q)$, and $g$ is in the $GL_n(q)$-class $\cC_{\bmu}$.  
\begin{enumerate}
\item[(i)] If $n \not\equiv 2($mod $4)$ or $q \not\equiv 3($mod $4)$, then $g$ is also real in $SL_n(q)$, and $\cC_{\bmu}$ is a union of real $SL_n(q)$-classes.

\item[(ii)] If $n \equiv 2($mod $4)$ and $q \equiv 3($mod $4)$, then $g$ is real in $SL_n(q)$ if and only if $g$ has some elementary divisor of the form $f(t)^k$ for $k$ odd.  That is, $\cC_{\bmu}$ is a union of real $SL_n(q)$-classes if and only if there is some $f \in \I$ such that $\bmu(f)$ has some odd part.
\end{enumerate}
\end{theorem}

Gill and Singh also answer the question of which elements are strongly real in $SL_n(q)$ in \cite[Theorem 6.1]{GiSi11}, which we state below.

\begin{theorem} [Gill and Singh] \label{slnstrong} Suppose that $g \in SL_n(q)$ such that $g$ is a real element of $SL_n(q)$, and $g$ is in the $GL_n(q)$-class $\cC_{\bmu}$.
\begin{enumerate}
\item[(i)] If $n \not\equiv 2($mod $4)$ or $q$ is even, then $g$ is also strongly real in $SL_n(q)$, and $\cC_{\bmu}$ is a union of strongly real $SL_n(q)$-classes.
\item[(ii)] If $n \equiv 2($mod $4)$ and $q$ is odd, then $g$ is strongly real in $SL_n(q)$ if and only if $g$ has an elementary divisor of the form $(t \pm 1)^k$ with $k$ odd.  That is, $\cC_{\bmu}$ is a union of strongly real $SL_n(q)$-classes if and only if at least one of $\bmu(t \pm 1)$ has an odd part.
\end{enumerate}

\end{theorem}

The main results of this paper are the statements on real and strongly real classes for $SU_n(q)$ analogous to Theorems \ref{slnreal} and \ref{slnstrong}, which are given in Theorems \ref{MainReal} and \ref{SUstrong}.

\section{Jordan decomposition and centralizers of unipotents} \label{JdecompCent}

\subsection{Jordan decomposition}

By the Jordan decomposition of elements, for any linear algebraic group $\bf{G}$ over $\bar{\F}_q$, any $g \in \bf{G}$ can be written uniquely as $g = su$ for commuting elements $s, u \in \bf{G}$, where $s$ is semisimple and $u$ is unipotent.  If $G$ is the group of $\F_q$-points of $\bf{G}$ under some Frobenius map and $g \in G$, then in the decomposition $g = su$, we also have $s, u \in G$, where $s$ is the $p'$-part of $g$ and $u$ is the $p$-part of $g$.

Now suppose that $G = GL_n(q)$ or $GU_n(q)$, and $g \in G$ is in the $G$-conjugacy class $\cC_{\bmu}$.  If we consider the Jordan decomposition $g = su$, then we can say exactly in which $G$-classes $s$ and $u$ lie in the next result, which follows directly from considering the uniqueness of the decomposition of $g$ in $GL_n$.

\begin{proposition} \label{Jdecomp}  Let $g \in G$ with $G = GL_n(q)$ or $GU_n(q)$, respectively, and suppose $g \in \cC_{\bmu}$.  Let $g = su$ be the Jordan decomposition of $g$.  Define $\bmu_s$ by $\bmu_s(f) = (1^{|\bmu(f)|})$ for $f \in \I$ or $\U$, respectively, and define $\mu_u$ to be the partition such that the multiplicity of each positive integer $i$ in $\mu_u$ is given by $m_i(\mu_u) = \sum_{f} \mathrm{deg}(f) m_i(\bmu(f))$.  Then $s$ is in the semisimple class $\cC_{\bmu_s}$, and $u$ is in the unipotent class corresponding to the partition $\mu_u$.
\end{proposition}

For example, suppose $g \in GL_{15}(q)$ and $g \in \cC_{\bmu}$ where $\bmu(f_1) = (2,1)$, $\bmu(f_2) = (2)$, and $\bmu(f_3) = (3,1)$ with $\mathrm{deg}(f_1) = 1$ and $\mathrm{deg}(f_2) = \mathrm{deg}(f_3) = 2$, and $\bmu(f) = \varnothing$ for all other $f \in \I$.  Then $g = su$, where $s \in \cC_{\bmu_s}$ with $\bmu_s(f_1) = (1^3)$, $\bmu_s(f_2) = (1^2)$, and $\bmu_s(f_3) = (1^4)$, and $u$ is in the unipotent class corresponding to the partition $\mu_u = (3^2 2^3 1^3)$.

\subsection{Centralizers of unipotent elements}

We first consider unipotent elements in $GL_n$, and we let $J_k$ denote the unipotent Jordan block of size $k$:
$$ J_k = \left( \begin{array}{ccccc} 1 & 1 &  &  &  \\   & 1 & 1 &  &  \\  &  &  \ddots & \ddots & \\  &  &  &  1 & 1 \\  &  &  &  & 1 \end{array} \right).$$
If a unipotent element has $m_k$ Jordan blocks of type $k$, then we denote this unipotent element by $\bigoplus_k J_k^{m_k}$, where the sum is over only those $k$ such that $m_k \neq 0$.  If $\mu = (1^{m_1} 2^{m_2} \cdots )$ is a partition of $n$, then $\bigoplus_k J_k^{m_k}$ is a representative element of the unipotent class of $GL_n$ corresponding to the partition $\mu$.  Taking $u = \bigoplus_k J_k^{m_k}$, consider the centralizer $C_{GL_n}(u)$ of $u$ in $GL_n$.  By \cite[Theorem 3.1(iv)]{LiSe12} (which the authors point out can be concluded from results in \cite{SpSt70, Wa62}), we have 
$$ C_{GL_n}(u) = UR,$$
where $U = \mathrm{R}_{\mathrm{u}}(C_{GL_n}(u))$ is the unipotent radical of the centralizer, and the other factor is reductive given by
$$ R \cong \prod_k GL_{m_k}.$$
We may further see precisely how the $GL_{m_k}$ factor embeds in the centralizer.  In particular, given any element $a = (a_{ij}) \in GL_{m_k}$, we see that the element $(a_{ij} I_k)$, consisting of $k$-by-$k$ scalar blocks replacing the entries of $a$, is in the centralizer of the element $J_k^{m_k}$.  We note that this also holds when we replace each Jordan block $J_k$ with some $GL_{k}$-conjugate.  One observation that we need is the fact that
\begin{equation} \label{det}
\mathrm{det}( (a_{ij} I_k) ) = (\mathrm{det}(a))^k.
\end{equation}

Now consider the finite group $GL_n(q)$, where we still consider the element $u = \bigoplus_k J_k^{m_k} \in GL_n(q)$ as a representative of the unipotent class corresponding to $\mu = (1^{m_1} 2^{m_2} \cdots )$.  By \cite[Theorem 7.1(ii)]{LiSe12}, the centralizer of $u$ in $GL_n(q)$ is given by $C_{GL_n(q)}(u) = U_F R_F$, where $U_F$ is a unipotent factor and $R_F$ is given by
$$ R_F \cong \prod_k GL_{m_k}(q),$$
given exactly by the $F$-fixed points of $R$, so that the $GL_{m_k}(q)$ factors are embedded the same as described above.  

In the case of the unitary group, the unipotent element $u = \bigoplus_k J_k^{m_k}$ of $GL_n$ is no longer necessarily in $GU_n(q)$.  However, we can choose our Frobenius map $F_{\eta}$ in such a way that there is a $GL_k$-conjugate of $J_k$,  $\tilde{J}_k \in GU_k(q)$, such that the unipotent element $\tilde{u} = \bigoplus_k \tilde{J}_k^{m_k}$ is a representative in $GU_n(q)$ of the unipotent class of type $\mu$.  As explained in \cite[p. 114]{LiSe12}, the Frobenius morphism $F_{\eta}$ can be chosen simultaneously to stabilize each $GL_{m_k}$ factor in $R$ of the centralizer $C_{GL_n}(\tilde{u})$.  The centralizer $C_{GU_n(q)}(\tilde{u})$ is then given by \cite[Theorem 7.1(ii)]{LiSe12} as $C_{GU_n(q)}(\tilde{u}) = U_{F_{\eta}} R_{F_{\eta}}$ with $U_{F_{\eta}}$ again a unipotent factor, and 
$$R_{F_{\eta}} \cong \prod_k GU_{m_k}(q),$$
given by the $F_{\eta}$-fixed points of the $R$ factor in $C_{GL_n}(\tilde{u})$, again embedded as before so that in particular (\ref{det}) holds.

The following result will be important in the main arguments of this paper.

\begin{lemma} \label{centdetsquare}
Let $u$ be a unipotent element in either $G = GL_n(q)$ or $G = GU_n(q)$ in the class corresponding to the partition $\mu = (1^{m_1} 2^{m_2} \cdots )$.  For any $\delta_k \in \F_q^{\times}$ or $C_{q+1}$, respectively, for all $k$ such that $m_k \neq 0$, there exists some $g \in C_G(u)$ such that $\mathrm{det}(g) = \prod_k \delta_k^k$.

Suppose that $m_k = 0$ when $k$ is odd, that is, $u$ has no elementary divisors of the form $(t-1)^k$ with $k$ odd.  Then for any element $g \in C_G(u)$, we have $\mathrm{det}(g)$ is a square for any $g \in C_G(u)$.
\end{lemma}
\begin{proof} From the discussion above, any $g \in C_G(u)$ is of the form $g = vr$ where $v$ is unipotent, and $r$ is in the factor isomorphic to either $\prod_k GL_{m_k}(q)$ or $\prod_k GU_{m_k}(q)$, respectively.   Since $v$ is unipotent, we have $\mathrm{det}(g) = \mathrm{det}(r)$.  Now write $r = \prod_k a_k$ with $a_k$ in the embedded factor of $GL_{m_k}(q)$ or $GU_{m_k}(q)$.  Then we have from (\ref{det}) that $\mathrm{det}(r) = \prod_k (\mathrm{det}(a_k))^k$, where the product is only over $k$ with $m_k \neq 0$.  Since any choice of $a_k \in GL_{m_k}(q)$ or $GU_{m_k}(q)$ will yield an element $g \in C_G(u)$, then given any prescribed elements $\delta_k \in \F_q^{\times}$ or $C_{q+1}$, we can choose $a_k$ such that $\mathrm{det}(a_k) = \delta_k$.  Then $\mathrm{det}(g) = \mathrm{det}(r) = \prod_k \delta_k^k$.

The assumption that $m_k = 0$ when $k$ is odd implies that for any choice of the $a_k$, and so for any $g \in C_G(u)$, we must have $\mathrm{det}(g)$ is a square.
\end{proof}

\subsection{An important example}

We briefly return to Theorem \ref{slnreal} of Gill and Singh.  In \cite{GiSi11}, the proofs of Theorem \ref{slnreal}(i) and the existence direction of (ii) are obtained by constructing explicit reversing elements $h \in SL_n(q)$ of the elements in question.  The more subtle argument is the converse statement of (ii), that is, if $n \equiv 2($mod $4)$ and $q \equiv 3($mod $4)$, and $g \in SL_n(q)$ is real in $GL_n(q)$ and has no elementary divisor of the form $f(t)^k$ with $k$ odd, then $g$ is not real in $SL_n(q)$.  In \cite[Proof of Proposition 5.5]{GiSi11}, it is proved that such an element has a reversing element $h \in GL_n(q)$ such that $\mathrm{det}(h) = -1$.  It is then stated that also there cannot exist reversing elements of $g$ in $GL_n(q)$ which have determinant $1$, although some of the details of the proof of this claim are omitted.  As similar arguments are crucial to the main results in the present paper, we fill in the details of this proof in this section.

We first consider the case that $g$ is unipotent, and so $g$ has no elementary divisor of the form $(t-1)^k$ with $k$ odd.  That is, $g$ is in the unipotent class corresponding to a partition $\mu = (1^{m_1} 2^{m_2} \cdots)$ such that $m_k = 0$ whenever $k$ is odd.  Given the reversing element $h \in GL_n(q)$ of $g$, if $h_1 \in GL_n(q)$ is any other reversing element of $g$, then we must have $h_1 = h a$ for some $a \in C_{GL_n(q)}(g)$.  By Lemma \ref{centdetsquare}, we have $\mathrm{det}(a)$ is a square in $\F_q^{\times}$, and $\mathrm{det}(h) = -1$ is a nonsquare in $\F_q^{\times}$ since $q \equiv 3($mod $4)$, and so $\mathrm{det}(h_1)$ is a nonsquare.  In particular, $\mathrm{det}(h_1) \neq 1$ and $h_1 \not\in SL_n(q)$.

In the case of general $g$, we consider the Jordan decomposition $g = su$ with $u$ unipotent.  We are assuming that $g$ has no elementary divisor of the form $f(t)^k$ with $k$ odd, so if $g \in \cC_{\bmu}$ then for every $f \in \I$ the partition $\bmu(f)$ has no odd parts.  It follows from Proposition \ref{Jdecomp} that $u$ has no elementary divisor of the form $(t-1)^k$ with $k$ odd, since the corresponding partition $\mu_u$ has no odd parts.  From the previous case, we know that $u$ is not real in $SL_n(q)$.  Now $g^{-1} = s^{-1} u^{-1}$, and if $ygy^{-1} = g^{-1}$ for some $y \in SL_n(q)$, then $g^{-1} = (ysy^{-1})(yuy^{-1})$ with $ysy^{-1}$ semisimple and $yuy^{-1}$ unipotent.  By the uniqueness of the Jordan decomposition, we would then have $yuy^{-1} = u^{-1}$, contradicting the fact that $u$ is not real in $SL_n(q)$.  Thus $g$ is not real in $SL_n(q)$.

\section{Initial observations} \label{initialobs}

We begin by recalling and adapting some of the results of Gill and Singh \cite{GiSi11} to use in our situation.  The following is a generalization of \cite[Lemma 4.2]{GiSi11} which has essentially the same proof, so we omit it here.

\begin{lemma} \label{realclasssplit} Let $G$ be a group and $N \lhd G$ a normal subgroup, with $g \in N$ in the $G$-conjugacy class $\cC$.  Suppose $\cC$ is the disjoint union $\cC = \cC_1 \cup \cdots \cup \cC_s$ where each $\cC_i$ is an $N$-conjugacy class.  Let $x_i \in G$ such that $x_i g x_i^{-1} \in \cC_i$ for each $i$.  Then $g$ is real (or strongly real) in $N$ if and only if $x_i g x_i^{-1}$ is real (or strongly real) in $N$ for each $i$.
\end{lemma}

Taking $G = GU_n(q)$ and $N = SU_n(q)$ in Lemma \ref{realclasssplit}, we see that for $g \in SU_n(q)$, to determine whether the $SU_n(q)$-class of $g$ is real or strongly real in $SU_n(q)$, it suffices to consider any $GU_n(q)$-conjugate of $g$.

Given a group $G$ and $g \in G$, define the subgroup $R_G(g)$ by
$$R_G(g) = \{ h \in G \, \mid \, hgh^{-1} = g \text{ or } g^{-1} \}.$$
Unless $g^2 = 1$, we have $[R_G(g) : C_G(g)] = 2$.  The following result is an adaptation of the first paragraph of \cite[Proof of Proposition 4.4]{GiSi11}.

\begin{proposition} \label{qeven}  Let $q$ be even and let $g \in SU_n(q)$ be real in $GU_n(q)$.  Then $g$ is real in $SU_n(q)$.
\end{proposition}
\begin{proof} Let $N = SU_n(q)$ and $G = GU_n(q)$.  If $g^2 = 1$, then $g$ is real in $N$, so we assume $g^2 \neq 1$.  Thus $[R_G(g) : C_G(g)]=2$.  If $g$ is not real in $N$, then $R_G(g) \cap N \leq C_G(g)$.  This implies $[R_G(g)N : N] = [R_G(g): R_G(g) \cap N]$ is even.  But this should divide $[G: N]$, which in turn divides the odd number $q+1$.  This contradiction implies $g$ must be real in $N$.
\end{proof}

In classifying the real classes of $SU_n(q)$, Proposition \ref{qeven} allows us to now restrict our attention to the case that $q$ is odd.  We may extend the method just used to obtain more information about real classes of $SU_n(q)$.

\begin{proposition} \label{oddpower}  Suppose that $q$ is odd, and that $g \in SU_n(q)$ is real in $GU_n(q)$.  If $g$ has any elementary divisor of the form $f(t)^k$ with $k$ odd, then $g$ is real in $SU_n(q)$.
\end{proposition}
\begin{proof}  Let $N = SU_n(q)$ and $G = GU_n(q)$, and suppose that $g$ is not real in $N$.  Then we have $R_G(g) \cap N \leq C_G(g)$, and $[R_G(g) N : C_G(g) N] = 2$.  Suppose that $g$ is in a $G$-conjugacy class of type $\nu = (\nu_1, \nu_2, \ldots, \nu_s)$ as defined in Section \ref{classes}, and let $t_{\nu} = \gcd(q+1, \nu_1, \ldots, \nu_s)$.  As discussed at the end of Section \ref{classes}, we have $t_{\nu} = [G : C_G(g) N]$, which must be even since $[R_G(g) N : C_G(g) N] = 2$.  Now every part $\nu_i$ of $\nu$ must be even, so writing $\nu = (1^{n_1} 2^{n_2} \cdots)$, we have that $n_k = 0$ whenever $k$ is odd.  By definition of $\nu$, this implies $g$ has no elementary divisor of the form $f(t)^k$ with $k$ odd.
\end{proof}

Recall that we let $C_{q+1}$ denote the cyclic subgroup of $\F_{q^2}^{\times}$ of order $q+1$.  The following is our analogue of \cite[Lemma 5.1]{GiSi11}.

\begin{lemma} \label{detsquare}
Let $q$ be odd and $g \in SU_n(q)$.  Then $g$ is real in $SU_n(q)$ if and only if there is a reversing element $h \in GU_n(q)$ such that $\mathrm{det}(h)$ has odd multiplicative order in $C_{q+1}$.  In particular, if $q \equiv 1($mod $4)$, then $g$ is real in $SU_n(q)$ if and only if there is a reversing element $h \in GU_n(q)$ such that $\mathrm{det}(h) = \lambda^2$ with $\lambda \in C_{q+1}$.
\end{lemma}
\begin{proof} If $g$ is real in $SU_n(q)$, then there is a reversing element with $\mathrm{det}(h) = 1$ and we are done.  Conversely, suppose $hgh^{-1} = g^{-1}$ with $h \in GU_n(q)$ such that $\mathrm{det}(h)$ has odd multiplicative order $c$.  Then $\mathrm{det}(h^c) = \mathrm{det}(h)^c = 1$, so $h^c \in SU_n(q)$.  Since $c$ is odd, we also have $h^c g (h^c)^{-1} = hgh^{-1} = g^{-1}$, and so $g$ is real in $SU_n(q)$.

The last statement follows from the fact that when $q \equiv 1($mod $4)$,  we have $(q+1)_2 = 2$, and so the elements of odd order in $C_{q+1}$ are exactly the squares.
\end{proof}

%
%
%

We now return to a result of Wonenburger \cite{Wo66} which states that for any vector space $V$ of finite dimension over a field of characteristic not $2$, any element of any orthogonal group $O(V)$ over $V$ is strongly real in $O(V)$.  We need a slight refinement of this statement.  Recall that in the finite orthogonal group $O^{\pm}_n(q)$, with $q$ odd, that for $g \in O^{\pm}_n(q)$, the elementary divisors of $g$ are such that $g$ is real in $GL_n(q)$.  Further, every elementary divisor of $g \in O^{\pm}_n(q)$ of the form $(t \pm 1)^{2k}$ appears with even multiplicity \cite[Sec. 2.6, Case (C)]{Wa62}.

\begin{lemma} \label{orthogdet} Let $q$ be odd and let $g \in O^{\pm}_n(q)$.  
\begin{enumerate}
\item[(1)] If $g$ has an elementary divisor of the form $(t \pm 1)^k$ with $k$ odd, then there exist reversing elements $h, h' \in O^{\pm}_n(q)$ such that $h^2 = h'^2 = 1$, and $\mathrm{det}(h) = -\mathrm{det}(h')$.
\item[(2)] If $g$ has no elementary divisor of the form $(t \pm 1)^k$ with $k$ odd, then $n$ is even, and $g$ has a reversing element $h \in O_n^{\pm}(q)$ such that $h^2 = 1$ and $\mathrm{det}(h) = (-1)^{n/2}$.
\end{enumerate}
\end{lemma}
\begin{proof} We let $V = \F_q^n$ with some underlying symmetric form $(\cdot, \cdot)$.  Then by \cite[Remark I after Lemma 5]{Wo66} $V$ can be decomposed as a direct sum of mutually orthogonal $g$-invariant subspaces, $V = \bigoplus_i V_i$, where we write $g_i = g|_{V_i}$, such that either $V_i$ is cyclic with respect to $g$ (and so with respect to $g_i$), or $V_i = U_{i1} \oplus U_{i2}$ where $U_{i1}$ and $U_{i2}$ are each cyclic with respect to $g_i$, and the minimal and characteristic polynomials of $g_i|_{U_{i1}}$ and $g_i|_{U_{i2}}$ are of the form $(t \pm 1)^{2n_i}$.  In each case we have $g_i \in O(V_i)$, where $O(V_i)$ is the orthogonal group with respect to $(\cdot, \cdot)$ restricted to $V_i$, and in the first case the characteristic polynomial of $g_i$ is self-reciprocal and not of the form $(t \pm 1)^{2n_i}$.  Wonenburger proves \cite[Lemmas 2 and 5]{Wo66} that each $g_i$ is reversed by some $h_i \in O(V_i)$ such that $h_i^2 = 1$, so that $h = \bigoplus_i h_i \in O(V)$ reverses $g$ and satisfies $h^2 = 1$.  We have $\mathrm{det}(h) = \prod_i \mathrm{det}(h_i)$.

In the case that $V_i$ is cyclic, either $g_i$ has characteristic polynomial of the form $(f \tilde{f})^k$, where $f \neq \tilde{f}$ and $f(t)$ is irreducible in $\F_q[t]$ so that $(f \tilde{f})^k$ has even degree, or of the form $f^k$ where $f = \tilde{f}$ is self-reciprocal and irreducible in $\F_q[t]$, in which case either $\mathrm{deg}(f)$ is even or $f(t) = t \pm 1$ and $k$ is odd.  Let $k_i$ be the degree of the characteristic polynomial of $g_i$, and let $v_i$ be a cyclic generator for $V_i$ so that a basis for $V_i$ is given by $\{ v_i, g_i v_i, \ldots, g_i^{k_i - 1} v_i \}$.  By \cite[Proof of Lemma 2]{Wo66}, $g_i$ is reversed by an $h_i \in O(V_i)$ such that $h_i^2 = 1$ defined by $h_i = 1_{P_i} \oplus -1_{Q_i}$ where $P_i$ is generated by vectors of the form $(g_i^l+ g_i^{-l})v_i$ and $Q_i$ is generated by vectors of the form $(g_i^l - g_i^{-l})v_i$.  In the case that $V_i = U_{i1} \oplus U_{i2}$ with $U_{i1}$ and $U_{i2}$ cyclic, let $u_{i1}$ and $u_{i2}$ be cyclic generators.  Then $h_i = 1_{P_i} \oplus -1_{Q_i}$ where $P_i$ is generated by vectors of the form $(g_i^l + g_i^{-l})u_{i1}$ and $(g_i^l + g_i^{-l})u_{i2}$ and $Q_i$ is generated by vectors of the form $(g_i^l - g_i^{-l})u_{i1}$ and $(g_i^l - g_i^{-l})u_{i2}$.

If $g$ has an elementary divisor of the form $(t \pm 1)^k$ with $k$ odd, then this corresponds to some $g_j$ with characteristic polynomial of the form $(t \pm 1)^{k_j}$ with $k_j = 2n_j -1$ for some $n_j$.  In this case, $P_j$ has dimension $n_j$ while $Q_i$ has dimension $n_j - 1$, and $h_j = 1_{P_j} \oplus -1_{Q_j}$ has determinant $(-1)^{n_j - 1}$.  However, $-h_j$ also reverses $g_j$ and satisfies $(-h_j)^2 = 1$, and $\mathrm{det}(-h_j) = - \mathrm{det}(h_j)$.  For any $h_i$ chosen for $i \neq j$, we can take $h = \bigoplus_i h_i$ and $h' = -h_j \oplus \bigoplus_{i \neq j} h_i$, so that one of $h$ or $h'$ has determinant $1$ and the other $-1$.

Now assume that $g$ has no elementary divisor of the form $(t \pm 1)^k$ with $k$ odd.  In the case that $V_i$ is not cyclic, $U_{i1}$ and $U_{i2}$ are each cyclic (although degenerate with respect to the underlying form), with $g_i|_{U_{i1}}$ and $g_i|_{U_{i2}}$ each having characteristic polynomial $(t \pm 1)^{2n_i}$.  Then $P_i$ and $Q_i$ each have dimension $2n_i$, so that $h_i = 1_{P_i} \oplus -1_{Q_i}$ has determinant $(-1)^{2n_i} = 1$, and we note that the characteristic polynomial of $g_i$ has degree $4n_i$.  

Now consider the case that $V_i$ is cyclic.  As discussed above, either $g_i$ has characteristic polynomial $(f_i \tilde{f_i})^{k_i}$ with $f_i \neq \tilde{f_i}$ or $f_i^{k_i}$ with $f_i = \tilde{f_i}$ and $f_i(t) \neq t \pm 1$.  In both cases, $\mathrm{dim} \; V_i$ is even, so that $n = \sum_i \mathrm{dim} \; V_i$ is now seen to be even.  In the first case, if $\mathrm{deg}(f_i) = d_i$, then let $n_i = d_i k_i$, so $h_i$ has determinant $(-1)^{n_i}$ and $n_i$ is the dimension of $Q_i$.  In the second case, we must have $\mathrm{deg}(f_i) = 2d_i$ even, and again we have $\mathrm{det}(h_i) = (-1)^{n_i}$ where $n_i = d_i k_i$ is the dimension of $Q_i$.  It now follows that $\mathrm{det}(h) = \prod_i \mathrm{det}(h_i) = (-1)^{n/2}$.
\end{proof}

We note that Lemma \ref{orthogdet} implies that if either $n$ is odd or $4|n$, every element of $SO^{\pm}_n(q)$ is strongly real in $SO^{\pm}_n(q)$, which was also mentioned by Gow \cite[p. 250]{Go85}.  We address the case that $n \equiv 2($mod $4)$ in Section \ref{strongreal}.

Consider a  conjugacy class $\cC_{\bmu}$ of $GU_n(q)$.  Let $\U_{\bmu}$ be the set of $f \in \U$ such that $\bmu(f) \neq \varnothing$.  For each $f \in \U_{\bmu}$, define $n_f = \mathrm{deg}(f) |\bmu(f)|$ and $\bmu_f: \U \rightarrow \cP$ by $\bmu_f(f) = \bmu(f)$ and $\bmu_f(f') = \varnothing$ if $f' \neq f$.  We may take an element $g \in \cC_{\bmu}$ which is in block diagonal form, $g = \oplus g_f$, where each $g_f \in GU_{n_f}(q)$ is in the class $\cC_{\bmu_f}$ of $GU_{n_f}(q)$.

In the case that $\cC_{\bmu}$ is a real class of $GU_n(q)$, so that $\bmu(f) = \bmu(\tilde{f})$ for each $f \in \U$, as in Section \ref{realclasses}, write $g$ as above in the form $g = \oplus g_{f^*}$, where we define $g_{f^*} = g_f$ and $n_{f^*} = n_f$ if $f = \tilde{f}$, and $g_{f^*} = g_f \oplus g_{\tilde{f}}$ with $n_{f^*} = n_f + n_{\tilde{f}}$ if $f \neq \tilde{f}$.  Then each $g_{f^*}$ is a real element of $GU_{n_{f^*}}(q)$.  Also note that when $f \neq t \pm 1$, then $n_{f^*}$ is always even, since $\mathrm{deg}(f)$ is even whenever $f = \tilde{f}$ and $f \neq t \pm 1$, and in general $\mathrm{deg}(f) = \mathrm{deg}(\tilde{f})$.

We may now apply Lemma \ref{orthogdet} in the following situation.

\begin{lemma} \label{nounipreal} 
Let $q$ be odd, and $\cC_{\bmu}$ a real class in $GU_n(q)$.  Let $g \in \cC_{\bmu}$ such that $g = \oplus g_{f^*}$ as above, and let ${g^* = \oplus_{f \neq t \pm 1} g_{f^*}}$ with $n^* = \sum_{f \neq t \pm 1} n_{f^*}$.  Then $g^* \in GU_{n^*}(q)$ is reversed by an element in $SU_{n^*}(q)$.
\end{lemma}
\begin{proof}  Since $g$ is real in $GU_n(q)$, we see that $g^*$ is real in $GU_{n^*}(q)$.  Also, we have $g^* \in SU_{n^*}(q)$ since $g^*$ is real and has no elementary divisor of the form $(t+1)^k$.  Now, if $g^*$ has an elementary divisor of the form $f(t)^k$ with $k$ odd, then the statement follows by Proposition \ref{oddpower}, so we may assume $g^*$ has no such elementary divisors.

If $f(t)^k$ is an elementary divisor of $g^*$ and $f$ is self-reciprocal, then $\mathrm{deg}(f)$ is even, while if $f$ is not self-reciprocal, then $\tilde{f}(t)^k$ is also an elementary divisor of $g^*$.  It follows that $4 | n^*$.  Since $g^*$ is real as an element of $GU_{n^*}(q)$ and has no powers of $t \pm 1$ as elementary divisors, it follows from the work of Wall \cite[Sec. 2.6, Cases (A) and (C)]{Wa62} that $g^*$ is an element of some orthogonal group $O^{\pm}_{n^*}(q)$ embedded in $GU_{n^*}(q)$.  Since $4 | n^*$, it follows from Lemma \ref{orthogdet} that $g^*$ is reversed by an element of the embedded $SO^{\pm}_{n^*}(q)$, and so $g^*$ is real in $SU_{n^*}(q)$.
\end{proof}

In light of Lemma \ref{nounipreal}, we are particularly interested in reversing elements of $g_{t-1}$ or $g_{t+1}$.  We therefore turn our attention to unipotent elements.

\section{Real unipotent elements} \label{RealUnis}

First we consider regular unipotent elements in $GU_n(q)$ (and so in $SU_n(q)$), where $q$ is odd.  Because of Proposition \ref{oddpower}, we may restrict our considerations to the case that $n$ is even.  

If we take our Hermitian matrix $\eta$ which defines $GU_n(q)$ to be
$$ \eta_n = \left( \begin{array}{ccc}   &  & 1 \\   & \adots &  \\ 1 &  &  \end{array} \right),$$
then we may find a regular unipotent element in $GU_n(q)$ which is upper unitriangular with all super-diagonal entries nonzero.  We fix $u$ to be such an element.

\begin{lemma} \label{reverseu} Let $q$ be odd, $n = 2m$ be even, and let $u \in GU_n(q)$ be the regular unipotent element just described.  Then any reversing element $h \in GU_n(q)$ of $u$ must be upper triangular with diagonal entries $(\beta, -\beta, \ldots, \beta, -\beta)$ such that $\beta \in \F_{q^2}^{\times}$ with $\beta^{q+1} = -1$, and such that $\mathrm{det}(h) = (-1)^m \beta^n$.
\end{lemma}
\begin{proof}  The fact that $h$ must be upper triangular follows from the fact that $u$ is upper unitriangular, and all super-diagonal entries of $u$ are nonzero.  Since $u^{-1}$ is also upper unitriangular and the super-diagonal entries are the negatives of those of $u$, it follows that the diagonal of $h$ must be of the form $(\beta, -\beta, \ldots, \beta, -\beta)$ for some $\beta \in \F_{q^2}^{\times}$.  It immediately follows that $\mathrm{det}(h) = (-1)^m \beta^n$.  Finally, since $h \in GU_n(q)$, we have $\eta_n^{-1}({^\top F(h)}^{-1}) \eta_n = h$.  By considering the first entry in the diagonal on each side, we have $-\beta^{-q} = \beta$, so $\beta^{q+1} = -1$. 
\end{proof}

We can first deal with the case $q \equiv 1($mod $4)$.  

\begin{lemma} \label{1mod4} Let $n = 2m$ be even and $q \equiv 1($mod $4)$.  If $u \in GU_n(q)$ is the regular unipotent element fixed above, then $u$ is real in $SU_n(q)$.
\end{lemma}
\begin{proof} Since $u$ is real in $GU_n(q)$, there is some reversing element $h \in GU_n(q)$.  By Lemma \ref{reverseu}, $h$ is upper triangular with diagonal entries $(\beta, -\beta, \ldots, \beta, -\beta)$ for some $\beta \in \F_{q^2}^{\times}$ such that $\beta^{q+1} = -1$, and $\mathrm{det}(h) = (-1)^m \beta^n$.  In particular, $\beta^2 \in C_{q+1}$.

If $n \equiv 0($mod $4)$, then write $n = 4k$.  Then $\mathrm{det}(h) = \beta^n = (\beta^{2k})^2$ with $\beta^{2k} \in C_{q+1}$.  By Lemma \ref{detsquare}, $u$ is real in $SU_n(q)$.

Now assume $n \equiv 2($mod $4)$, and write $n = 4k+2 = 2m$.  We have $\mathrm{det}(h) =  -\beta^{2m}$.   Let $\gamma$ be a multiplicative generator for $C_{q+1}$.  Since $(q+1)_2 = 2$, then $-1 = \gamma^{(q+1)/2}$ is not the square of an element in $C_{q+1}$.  We have $\beta^2 \in C_{q+1}$, so $\beta^2 = \gamma^s$ for some integer $s$.  If $s$ is even, then  
 $\beta \in C_{q+1}$, which contradicts $\beta^{q+1} = -1$.  Thus since $m=2k+1$ is odd, we see $\beta^{2m}$ is some odd power of $\gamma$, which implies $-\beta^{2m}$ is the square of some element in $C_{q+1}$.  By Lemma \ref{detsquare}, $u$ is real in $SU_n(q)$.
\end{proof}

The case $q \equiv 3($mod $4)$ is slightly different.

\begin{lemma} \label{3mod4} Let $n=2m$ be even and $q \equiv 3($mod $4)$.  If $u \in GU_n(q)$ is the regular unipotent fixed above, then $u$ is real in $SU_n(q)$ if and only if $n$ is divisible by $(q^2 - 1)_2$.
\end{lemma}
\begin{proof} Again, $u$ is real in $GU_n(q)$, and so by Lemma \ref{reverseu} any reversing element $h \in GU_n(q)$ of $u$ must be upper triangular with diagonal entries $(\beta, -\beta, \ldots, \beta, -\beta)$ for some $\beta \in \F_{q^2}^{\times}$ such that $\beta^{q+1} = -1$, and $\mathrm{det}(h) = (-1)^m \beta^n$.  If $|\beta|$ is the multiplicative order of $\beta$, then $\beta^{q+1} = -1$ implies that $(|\beta|)_2 = 2(q+1)_2 = (q^2 - 1)_2$.  By Lemma \ref{detsquare}, we know $u$ is real if and only if a reversing $h \in GU_n(q)$ exists such that $\mathrm{det}(h) = (-1)^n \beta^{2m}$ has odd order in $\F_{q^2}^{\times}$.  However, the fact that $(|\beta|)_2 = 2(q+1)_2$ means that this happens if and only if $m$ is divisible by $(q+1)_2$.  Thus $u$ is real in $SU_n(q)$ if and only if $n$ is divisible by $(q^2 - 1)_2$.
\end{proof}

We can now describe which unipotent elements are real in $SU_n(q)$ in general.

\begin{lemma} \label{uni1mod4} Let $u$ be any unipotent element in $GU_n(q)$, and suppose $q \equiv 1($mod $4)$.  Then $u$ is real in $SU_n(q)$.
\end{lemma}
\begin{proof}  If $u$ has any elementary divisor of the form $(t-1)^k$ with $k$ odd, then $u$ is real in $SU_n(q)$ by Proposition \ref{oddpower}, so we may assume otherwise.  Suppose $u$ is in the unipotent class of $GU_n(q)$ corresponding to the partition $\mu$ of $n$, so we are assuming $\mu$ has no odd parts, and so $\mu = (\mu_1, \mu_2, \ldots) = (2^{m_2} 4^{m_4} \cdots)$.

Define $GU_n(q)$ by the Hermitian matrix in block diagonal form $\bigoplus_i \eta_{\mu_i}$, where each $\eta_{\mu_i}$ has $1$'s on the antidiagonal and $0$'s elsewhere as above.  Then there is an element in the conjugacy class corresponding to $\mu$ of the block diagonal form $\bigoplus_k \tilde{J}_k^{m_k}$, where each $\tilde{J}_k$ is regular unipotent in $GU_k(q)$ of the form considered in Lemma \ref{1mod4}.  Since each $\tilde{J}_k$ is real in $SU_k(q)$, then  $\bigoplus_k \tilde{J}_k^{m_k}$ is real in $SU_n(q)$, as is $u$.
\end{proof}

Again, the case $q \equiv 3($mod $4)$ is more complicated.  We consider elements which are slightly more general than unipotents.

\begin{lemma} \label{3mod4pmuni}  Let $g \in SU_n(q)$ such that $g$ only has elementary divisors which are powers of $t \pm 1$, and suppose $q \equiv 3($mod $4)$.  Then $g$ is real in $SU_n(q)$ if and only if one of the following holds:
\begin{enumerate}
\item[(1)] $g$ has an elementary divisor of the form $(t \pm 1)^k$ with $k$ odd.
\item[(2)] $(q^2 - 1)_2$ divides $k$ for all elementary divisors $(t \pm 1)^k$ of $g$.
\item[(3)] Write $2^r$ for the smallest power of $2$ such that $2 \leq 2^r < (q^2 - 1)_2$ and $(k)_2 = 2^r$ for some elementary divisor $(t \pm 1)^k$ of $g$.  Then $g$ has an even number of elementary divisors of the form $(t \pm 1)^{e}$ such that $(e)_2 = 2^r$.
\end{enumerate}
\end{lemma}
\begin{proof}  Note that if (1) holds, then $g$ is real in $SU_n(q)$ by Proposition \ref{oddpower}.  So we now assume (1) does not hold.  Then $g$ is in the $GU_n(q)$-conjugacy class corresponding to $\bmu$, where we let 
$$\bmu(t -1) = \tau = (\tau_1, \tau_2, \ldots) = (2^{m_2(\tau)} 4^{m_4(\tau)} \cdots ),$$ and
$$ \bmu(t+1) = \nu = (\nu_1, \nu_2, \ldots) = (2^{m_2(\nu)} 4^{m_4(\nu)} \cdots),$$
so that each $\tau_i$ and each $\nu_i$ is even.  

By Lemma \ref{realclasssplit}, we may assume that $g$ is of the form $g = u \oplus v$, where $u = g_{t-1}$ and $u = g_{t+1}$ as in Section \ref{initialobs} before Lemma \ref{nounipreal}.  Note that $-v$ is unipotent.  

Similar to the proof of Lemma \ref{uni1mod4}, we define $GU_n(q)$ by the Hermitian matrix in block diagonal form $\bigoplus_i \eta_{\tau_i} \oplus \bigoplus_i \eta_{\nu_i}$.  Then there is an element in the conjugacy class corresponding to $\bmu$ in the block diagonal form $\bigoplus_k \tilde{J}_k^{m_k(\tau)} \oplus \bigoplus_k (-\tilde{J}_k)^{m_k(\nu)}$, where each $\tilde{J}_k$ is regular unipotent of the form considered in Lemma \ref{3mod4}.  If we assume (2) holds, then each $\tilde{J}_k$ and $-\tilde{J}_k$ is real in $SU_k(q)$, which implies $\bigoplus_k \tilde{J}_k^{m_k(\tau)} \oplus \bigoplus_k(-\tilde{J}_k)^{m_k(\nu)}$, and thus $g$, is real in $SU_n(q)$.

We now also assume that (2) does not hold, and we show that $g$ is real in $SU_n(q)$ if and only if (3) holds.  As in the previous paragraph, we consider the element which is $GU_n(q)$-conjugate to $g$ defined by $\tilde{g} = \tilde{u} \oplus \tilde{v} = \bigoplus_k \tilde{J}_k^{m_k(\tau)} \oplus \bigoplus_k(-\tilde{J}_k)^{m_k(\nu)}$.  

Let $m_k=m_k(\tau)+m_k(\nu)$ and for each $k$ such that $m_k \neq 0$, fix $h_k \in GU_k(q)$ to be a reversing element of $\tilde{J}_k$, so that by Lemma \ref{reverseu}, each $h_k$ is upper triangular with diagonal entries $(\beta_k, -\beta_k, \ldots, \beta_k, -\beta_k)$ for some $\beta_k \in \F_{q^2}^{\times}$ satisfying $\beta_k^{q+1} = -1$.  It follows that $h_k$ is also a reversing element for $-\tilde{J}_k$.  Now the block diagonal element $h' = \bigoplus_k \bigoplus_{i=1}^{m_k} h_k \in GU_n(q)$ is a reversing element for $\tilde{g}$ in $GU_n(q)$, and we have $\mathrm{det}(h') = (-1)^{n/2} \prod_k \beta_k^{k m_k}$.  Now every reversing element of $\tilde{g}$ in $GU_n(q)$ is of the form $h = h' c$, where $c$ can be any element of $C_{GU_n(q)}(\tilde{g})$.  Now, $C_{GU_n(q)}(\tilde{g})$ is given by the direct sum
$$C_{GU_n(q)}(\tilde{g}) = C_{GU_{|\tau|}(q)}(\tilde{u}) \oplus C_{GU_{|\nu|}(q)}(\tilde{v}).$$
Since $-\tilde{v}$ is unipotent, and has the same centralizer as $\tilde{v}$, we may apply Lemma \ref{centdetsquare} to both $\tilde{u}$ and $\tilde{v}$.  In particular, the possible determinants of elements in $C_{GU_n(q)}(\tilde{g})$ are of the form $\prod_k \delta_k^k$, where the product is over all $k$ which appear as parts of either $\tau$ or $\nu$, and $\delta_k \in C_{q+1}$ are arbitrary.   It follows that all possible determinants of reversing elements $h$ of $\tilde{g}$ in $GU_n(q)$ are of the form
$$ \mathrm{det}(h) = (-1)^{n/2} \prod_k \beta_k^{k m_k} \delta_k^k.$$

Now let $2^r$ be the smallest power of $2$ such that $2 \leq 2^r < (q^2-1)_2$ and $(k)_2 = 2^r$ for some elementary divisor $(t\pm1)^k$ of $\tilde{g}$ (and so of $g$).  Noting that all of the parts of $\tau$ and $\nu$ are assumed even, we re-label all of the distinct part sizes of $\tau$ or $\nu$ as $2k_i$ such that $2^{r-1} = (k_1)_2 \leq (k_2)_2 \leq \cdots$.  We also let $s$ be the total number of distinct part sizes of $\tau$ or $\nu$ with $2$-part $2^r$, so $(k_1)_2 = (k_2)_2 = \cdots = (k_s)_2 = 2^{r-1}$.  We show that $g$ is real in $SU_n(q)$ if and only if $\sum_{i=1}^s m_{2k_i}$ is even, where $m_{2k_i} = m_{2k_i}(\tau) + m_{2k_i}(\nu)$. 

Let $\alpha \in C_{q+1}$ with multiplicative order $|\alpha| = (q+1)_2$.  Since for each $\beta_{2k_i}$ we have $|\beta_{2k_i}|$ is divisible by $(q^2 - 1)_2$, we can write $\beta^2_{2k_i} = \alpha^{\sigma_i} \gamma_i$ for some odd integer $\sigma_i$ and some $\gamma_i \in C_{q+1}$ with odd multiplicative order.  We also write each $\delta_{2k_i} = \alpha^{\rho_i} \gamma_i'$ where $\gamma_i' \in C_{q+1}$ has odd multiplicative order and $\rho_i$ is some integer ($\delta_{2k_i} \in C_{q+1}$ is still arbitrary with freedom to choose $\rho_i$ and $\gamma_i'$).  Now
$$ \mathrm{det}(h) = (-1)^{n/2} \alpha^{\sum_i k_i(m_{2k_i} \sigma_i + 2 \rho_i)} \prod_i \gamma_i'',$$
where $\gamma_i'' = \gamma_i^{m_{2k_i} k_i} \gamma_i'^{2k_i}$ has odd multiplicative order, and so $\prod_i \gamma_i'' = \gamma$ does as well.

Now write
$$ x = \sum_{i=1}^s (k_i)_{2'} (m_{2k_i} \sigma_i + 2\rho_i) \quad \text{ and } \quad y = \sum_{i > s} \frac{k_i}{(k_1)_2} (m_{2k_i} \sigma_i + 2\rho_i),$$
so that we have
\begin{equation} \label{revdet}
\mathrm{det}(h) = (-1)^{n/2} \alpha^{(k_1)_2 (x+y)} \gamma,
\end{equation} 
where $y$ is necessarily even.

First assume that $\sum_{i=1}^s m_{2k_i}$ is odd.  Since each $(k_i)_{2'}$ and each $\sigma_i$ is odd, this implies that $x$ is odd, and so $x+y$ is odd.  Suppose that there is some $h$ as above such that $\mathrm{det}(h)$ has odd multiplicative order.  If $k_1$ is even, note first that this means $4|n$.  Since $\gamma$ has odd multiplicative order, then the only way $\mathrm{det}(h)$ can as well is if $\alpha^{(k_1)_2(x+y)}$ does.  Since $x+y$ is odd, this would mean $\alpha^{(k_1)_2}$ has odd multiplicative order.  But $|\alpha| = (q+1)_2$, while $(k_1)_2 < (q+1)_2$, a contradiction.  So suppose $k_1$ is odd, so $2k_i \equiv 2($mod $4)$ for $i \leq s$.  Since $\sum_{i=1}^s m_{2k_i}$ is odd, then $n/2$ is odd.  The only way $\mathrm{det}(h)$ can have odd multiplicative order is if $(-1)^{n/2} \alpha^{(k_1)_2 (x+y)} = - \alpha^{x+y}$ does.  But $x+y$ is odd, while $|\alpha| = (q+1)_2 \geq 4$, and so this is impossible.  Thus no such $h$ exists, and $g$ cannot be real in $SU_n(q)$ by Lemma \ref{detsquare}.

Now assume that $\sum_{i=1}^s m_{2k_i}$ is even.  In this case, we must have $4|n$.  We show that $\rho_i$ and $\gamma_i'$ can be chosen (that is, $c \in C_{GU_n(q)}(\tilde{g})$ can be chosen) so that $\mathrm{det}(h)$ has odd multiplicative order in $C_{q+1}$.  We fix $\gamma_i'$ to be some elements of odd multiplicative order in $C_{q+1}$, and we take $\rho_i = 0$ for $i > 1$.  We will choose $\rho_1$ such that 
$$ x + y =  (k_1)_{2'}( m_{2k_1} \sigma_1 + 2\rho_1 )+ \sum_{i > 1} \frac{k_i}{(k_1)_2} m_{2k_i} \sigma_i $$
is divisible by $(q+1)_2$, which will give $\alpha^{(k_1)_2 (x + y)} = 1$.  Note that for $i \leq s$ we have $k_i/(k_1)_2$ is odd while $k_i/(k_1)_2$ is even for $i > s$.  Since $(k_1)_{2'}$ and all $\sigma_i$ are odd, and $\sum_{i=1}^s m_{2k_i}$ is assumed even, then we have
$$ (k_1)_{2'} m_{2k_1} \sigma_1 + \sum_{i > 1} \frac{k_i}{(k_1)_2} m_{2k_i} \sigma_i $$ 
must be even.  Noting that $(k_1)_{2'}$ is invertible modulo $(q+1)_2$, we now choose $\rho_1$ such that
$$\rho_1 \equiv \frac{1}{2} \left(-(k_1)_{2'} m_{2k_1} \sigma_1 - \sum_{i>1} \frac{k_i}{(k_1)_2} m_{2k_i} \sigma_i \right) (k_1)_{2'} ^{-1}  (\text{mod } (q+1)_2 ).$$
This gives $\alpha^{(k_1)_2(x+y)} = 1$ as claimed, which makes $\mathrm{det}(h) = \gamma$ have odd multiplicative order in $C_{q+1}$.  By Lemma \ref{detsquare}, we now have $g$ is real in $SU_n(q)$.
\end{proof}

We restate the conditions in Lemma \ref{3mod4pmuni} for the unipotent case in a slightly refined manner as follows.

\begin{lemma} \label{uni3mod4}  Let $u \in GU_n(q)$ be unipotent, and suppose $q \equiv 3($mod $4)$.  Then $u$ is real in $SU_n(q)$ if and only if the following holds:

\noindent Write $2^r$ for the smallest power of $2$ such that $2 \leq 2^r < (q^2 - 1)_2$ and $(k)_2 = 2^r$ for some elementary divisor $(t-1)^k$ of $u$.  Then $u$ has an even number of elementary divisors of the form $(t - 1)^{e}$ such that $(e)_2 = 2^r$ (where the number of these is $0$ if the smallest power $2^r$ dividing the exponent of some elementary divisor $(t-1)^k$ is either $1$ or is at least $(q^2 - 1)_2$).
\end{lemma}

The conditions implied by Lemma \ref{3mod4pmuni} for the case $n \equiv 2($mod $4)$ are a bit simpler.

\begin{corollary} \label{uni2mod43mod4}  Let $u \in GU_n(q)$ be unipotent and suppose $q \equiv 3($mod $4)$ and $n \equiv 2($mod $4)$.  Then $u$ is real in $SU_n(q)$ if and only $u$ has an elementary divisor of the form $(t-1)^k$ with $k$ odd.
\end{corollary}
\begin{proof} If $u$ has such an elementary divisor, then we know $u$ is real in $SU_n(q)$.  Suppose $u$ has no such elementary divisor.  Then all elementary divisors of $u$ are of the form $(t-1)^k$ with $k$ even, and there must exist some $k$ with $(k)_2 = 2$, since $n \equiv 2($mod $4)$.  Since $(q^2 - 1)_2 \geq 4$, then we know $(q^2 -1)_2$ does not divide such $k$.  Now, $2$ is the smallest $2$-part of a $k$ such that $(t-1)^k$ is an elementary divisor of $u$.  There must be an odd number of such elementary divisors since $n \equiv 2($mod $4)$.  Since conditions (1), (2), and (3) of Lemma \ref{3mod4pmuni} all fail, it follows $u$ is not real in $SU_n(q)$.
\end{proof}

\section{Classification of real elements} \label{MainThm}

The following is a key result to understand reality in $SU_n(q)$ in the most complicated case.

\begin{lemma} \label{keyjordanlemma} Suppose that $q \equiv 3($mod $4)$ and $n \equiv 0($mod $4)$.  Let $g \in SU_n(q)$ with Jordan decomposition $g = su$.  Then $g$ is real in $SU_n(q)$ if and only if $u$ is real in $SU_n(q)$.
\end{lemma}
\begin{proof} First, if $g$ is real in $SU_n(q)$ and is reversed by $h$, then $g^{-1} = s^{-1} u^{-1} = (h s h^{-1})(h u h^{-1})$, and $u^{-1} = huh^{-1}$ by uniqueness of the Jordan decomposition, so $u$ is real in $SU_n(q)$.  Therefore we now assume that $u$ is real in $SU_n(q)$, and we prove that $g$ is real in $SU_n(q)$.

We may suppose that $g$ is in the form $g = \oplus g_{f^*}$, as in Section \ref{initialobs} before Lemma \ref{nounipreal}.  Write 
$$g^* = \oplus_{f \neq t \pm 1} g_{f^*}, \quad \text{ so } \quad g = g^* \oplus g_{t + 1} \oplus g_{t - 1},$$
where, as in Lemma \ref{nounipreal}, $g^* \in SU_{n^*}(q)$ with $n^* = \sum_{f \neq t \pm 1} n_{f^*}$.  By Lemma \ref{nounipreal}, $g^*$ is real in $SU_{n^*}(q)$.  Write $g_{\pm} = g_{t+1} \oplus g_{t-1}$ and $n_{\pm} = n_{t+1} + n_{t-1}$, so $g_{\pm} \in SU_{n_{\pm}}(q)$.  Since $g^*$ is real in $SU_{n^*}(q)$, we may assume that $g_{\pm}$ is not real in $SU_{n_{\pm}}(q)$, as otherwise we have $g$ is real in $SU_n(q)$.

First note that if $u$ has an elementary divisor of the form $(t - 1)^k$ with $k$ odd, then by Proposition \ref{Jdecomp}, $g$ must have an elementary divisor of the form $f(t)^k$ with $k$ odd, and so by Proposition \ref{oddpower} $g$ is real in $SU_n(q)$.  Also, if $(q^2 - 1)_2$ divides $k$ for all elementary divisors of the form $(t - 1)^k$ of $u$, then by Proposition \ref{Jdecomp}, it follows that $(q^2 - 1)_2$ divides $k$ for all elementary divisors $f(t)^k$ of $g$.  In particular, $(q^2 - 1)_2$ divides $k$ for all elementary divisors $(t \pm 1)^k$ of $g_{t+1} \oplus g_{t-1}$.  By Lemma \ref{3mod4pmuni}, we have $g_{\pm} = g_{t+1} \oplus g_{t-1}$ is real in $SU_{n_{\pm}}(q)$, and since $g^*$ is real in $SU_{n^*}(q)$, it follows that $g = g^* \oplus g_{\pm}$ is real in $SU_n(q)$.  So now let $2^r$ be the smallest power of $2$ such that $2 \leq 2^r < (q^2 -1)_2$ and $(k)_2 = 2^r$ for some elementary divisor $(t-1)^k$ of $u$.  By Lemma \ref{uni3mod4}, we may now assume that $u$ has a positive even number of elementary divisors of the form $(t-1)^e$ satisfying $(e)_2 = 2^r$.  

By Proposition \ref{Jdecomp}, $2^r$ is also the smallest power of $2$ such that $2 \leq 2^r < (q^2 - 1)_2$ and $(k)_2 = 2^r$ for some elementary divisor $f(t)^k$ of $g$.  Suppose that $g$ has an elementary divisor of the form $(t \pm 1)^k$ such that $(k)_2 = 2^r$.  As noted in Section \ref{initialobs}, if $f(t)$ is any elementary divisor of $g$ of the form $f(t)^e$ such that $(e)_2 = 2^r$ and $f(t) \neq t \pm 1$, then $\mathrm{deg}(f)$ is even if $f = \tilde{f}$,  while if $f \neq \tilde{f}$, then $\tilde{f}(t)^e$ is an elementary divisor of $g$ with the same multiplicity as $f(t)^e$.   Since the total number of elementary divisors of $u$ of the form $(t -1)^k$ with $(k)_2 = 2^r$ is even, Lemma \ref{Jdecomp} then implies that the total number of elementary divisors of $g$ of the form $(t \pm 1)^k$ with $(k)_2 = 2^r$ is also even.  But now $g_{\pm}$ is real in $SU_{n_{\pm}}(q)$ by Lemma \ref{3mod4pmuni}, and then $g$ is real in $SU_n(q)$.  Thus, we can assume that $g$ has no elementary divisor of the form $(t \pm 1)^k$ with $(k)_2 = 2^r$, but there is some elementary divisor of $g$ of the form $f(t)^k$ with $(k)_2 = 2^r$ and $f(t) \neq t \pm 1$.

Let $(k_1)_2$ be the smallest $2$-part of an exponent of an elementary divisor of $g_{\pm}$.  From the previous paragraph, we are assuming $2^r < (k_1)_2$.  Since we are assuming that $g_{\pm}$ is not real in $SU_{n_{\pm}}(q)$, then the number of elementary divisors of $g_{\pm}$ with exponent having $2$-part equal to $(k_1)_2$ is odd, by Lemma \ref{3mod4pmuni}.  Note also that as in the proof of Lemma \ref{nounipreal}, we have $4|n^*$, and since $4|n$ and $n = n^* + n_{\pm}$ then $4|n_{\pm}$.  We know that $g_{\pm}$ is real in $GU_{n_{\pm}}(q)$, so let $h_{\pm} \in GU_{n_{\pm}}(q)$ be any reversing element of $g_{\pm}$.  If $g$ is in the $GU_n(q)$-class corresponding to $\bmu$, let $\bmu(t-1) = \tau$ and $\bmu(t+1) = \nu$.  Let $\alpha \in C_{q+1}$ have multiplicative order $|\alpha| = (q+1)_2$.  Using exactly the notation as in the proof of Lemma \ref{3mod4pmuni}, we know from Equation (\ref{revdet}) that the determinant of $h_{\pm}$ must be of the form
$$ \mathrm{det}(h_{\pm}) = (-1)^{n_{\pm}/2} \alpha^{(k_1)_2(x+y)} \gamma,$$
where $\gamma$ has odd multiplicative order.  Since we are assuming that $g_{\pm}$ is not real in $SU_{n_{\pm}}(q)$, then we know from the proof of Lemma \ref{3mod4pmuni} that $x+y$ is odd.  Since $|\alpha| = (q+1)_2$, then $\alpha_1 = \alpha^{x+y}$ also has multiplicative order $|\alpha_1| = (q+1)_2$.  Since $4|n_{\pm}$, then we have
$$ \mathrm{det}(h_{\pm}) = \alpha_1^{(k_1)_2} \gamma.$$

Now let $h^* \in SU_{n^*}(q)$ be a reversing element for $g^*$.  If $c \in C_{GU_{n^*}(q)}(g^*)$, then $h^*c$ is a reversing element for $g^*$ in $GU_{n^*}(q)$.  Then $h^*c \oplus h_{\pm}$ is a reversing element in $GU_n(q)$ for $g = g^* \oplus g_{\pm}$, and 
$$ \mathrm{det}(h^*c \oplus h_{\pm}) = \mathrm{det}(h^*c) \mathrm{det}(h_{\pm}) = \mathrm{det}(c) \alpha_1^{(k_1)_2} \gamma.$$
Since $\gamma$ has odd multiplicative order, if we can find $c \in C_{GU_{n^*}(q)}(g^*)$ such that $\mathrm{det}(c) \alpha_1^{(k_1)_2}$ has odd multiplicative order in $C_{q+1}$, then by Lemma \ref{detsquare} we have $g$ is real in $SU_n(q)$.

It follows from the work of Wall \cite[Sec. 2.6, Case (A)]{Wa62} (see also \cite[Proposition 2.1]{GSV14}), that for $g^* = \oplus_{f \neq t \pm 1} g_f$, we have
$$ C_{GU_{n^*}(q)}(g^*) = \bigoplus_{f \neq t \pm 1} C_{GU_{n_f}(q)}(g_f).$$
So, for any $c \in C_{GU_{n^*}(q)}(g^*)$, we write $c = \oplus_{f \neq t \pm 1} c_f$, where $c_f \in C_{GU_{n_f}(q)}(g_f)$.  Now let $f_1$ be the polynomial such that $f_1(t) \neq t \pm 1$ and $g$ has as an elementary divisor $f_1(t)^k$ with $(k)_2 = 2^r < (k_1)_2$.  

Recall that $(k)_{2'}$ is invertible modulo $(q+1)_2$ and that $(k)_2$ divides $(k_1)_2$, so that we may define $\delta=\alpha_1^{-(k_1)_2/k}\in C_{q+1}$.  We will take $c$ such that $c_f = 1$ when $f \neq f_1$ and show that we may find $c_{f_1} \in C_{GU_{n_{f_1}}(q)}(g_{f_1})$ such that $\det(c_{f_1})=\delta^k$, so that $\mathrm{det}(c) \alpha_1^{(k_1)_2} = \mathrm{det}(c_{f_1}) \alpha_1^{(k_1)_2}=\delta^k\alpha_1^{(k_1)_2}$ has odd multiplicative order in $C_{q+1}$.

Write $g_{f_1}=s_1u_1$ in terms of its Jordan decomposition.  In the following, we make use of the structure of the centralizer of semisimple elements, details of which can be found in \cite[Proposition (1A)]{FoSr82}, \cite[Part (B1), Proof of Theorem 4.1]{TiZa96}, and \cite[Theorem 2.4.1]{SF13}, for example.  

Write $d$ for the degree of $f_1$ and $C(s_1)$ for the centralizer $C_{GU_{n_{f_1}}(q)}(s_1)$.  From Lemma \ref{Jdecomp}, $s_1$ is in the class corresponding to $(1^{|\bmu(f_1)|})$.  Then $C(s_1)\cong GU_{|\bmu(f_1)|}(q^d)$ if $f_1=f_1^\checkmark$ and $C(s_1)\cong GL_{|\bmu(f_1)|}(q^d)$ if $f_1= f_2 f_2^\checkmark$ for some irreducible $f_2$ with $f_2\neq f_2^\checkmark$, where $d=2\deg(f_2)$.  

As in  \cite[discussion before Lemma 5.2]{GiSi11}, in identifying $C(s_1)$ with $GU_{|\bmu(f_1)|}(q^d)$, respectively $GL_{|\bmu(f_1)|}(q^d)$, we will need to distinguish between the determinant of an element of $C(s_1)$ when viewed over $\F_{q^{2d}}$, respectively $\F_{q^d}$, and when viewed over $\F_{q^2}$.  In the case $f_1=f_1^\checkmark$, if $z\in C(s_1)$  has determinant $\det_{q^{2d}}(z)$ as an element of $GU_{|\bmu(f_1)|}(q^d)$, then $z$ has determinant $\det_{q^2}(z)=\mathrm{N}_{q^{2d}/q^2}(\det_{q^{2d}}(z))$ as an element of $GU_{n_{f_1}}(q)$.   In the case $f_1=f_2f_2^\checkmark$, if $z\in C(s_1)$ has determinant $\det_{q^d}(z)$ as an element of $GL_{|\bmu(f_1)|}(q^d)$, then $z$ has determinant $\det_{q^2}(z)=\mathrm{N}_{q^{d}/q^2}(\det_{q^d}(z)^{1-q})$ as an element of $GU_{n_{f_1}}(q)$.   Here $\mathrm{N}_{q^{2d}/q^2}$ and $\mathrm{N}_{q^{d}/q^2}$ are the standard norm maps for finite fields. 

Now, recalling that the norm maps are surjective, we may find $\tilde{\delta}$ such that in the case $f_1=f_1^\checkmark$, $\tilde{\delta}\in C_{q^d+1}\leq \F_{q^{2d}}^\times$ and $\mathrm{N}_{q^{2d}/q^2}(\tilde{\delta})=\delta$,  and in the case $f_1=f_2f_2^\checkmark$, $\tilde{\delta}\in \F_{q^d}^\times$ and $\mathrm{N}_{q^d/q^2}(\tilde{\delta})^{1-q}=\delta$.

Note that an element is a member of $C_{GU_{n_{f_1}}(q)}(g_{f_1})$ if and only if it is a member of $C(s_1)$ which commutes with $u_1$.  Write $\bmu(f_1)=\rho=(\rho_1,\rho_2,\ldots)$ such that $\rho_\ell = k$, so $(\rho_\ell)_2=2^r$.  Then $u_1$ is conjugate in $C(s_1)$ (identified with $GU_{|\bmu(f_1)|}(q^d)$, respectively $GL_{|\bmu(f_1)|}(q^d)$) to an element $\widetilde{u_1}$ in block diagonal form corresponding to its elementary divisors $(t-1)^{\rho_i}$.  The element
\[\widetilde{c_{f_1}} = \bigoplus_{i<\ell} I_{\rho_i} \oplus \tilde{\delta}I_{\rho_\ell}\oplus\bigoplus_{i>\ell} I_{\rho_{i}},\] where $I_{\rho_i}$ is the $\rho_i$-by-$\rho_i$ identity, is then an element of $GU_{|\bmu(f_1)|}(q^d)$, respectively $GL_{|\bmu(f_1)|}(q^d)$, which commutes with $\widetilde{u_1}$.  Hence there is a conjugate in $C(s_1)$, say  $c_{f_1}$, of $\widetilde{c_{f_1}}$ which commutes with $u_1$, and we see that $\det_{q^{2d}}(c_{f_1})=\widetilde{\delta}^{\rho_\ell}$, respectively $\det_{q^{d}}(c_{f_1})=\tilde{\delta}^{\rho_\ell}$, so that $\det_{q^2}(c_{f_1})=\delta^{\rho_\ell}$ by our definition of $\tilde{\delta}$.  Hence $c_{f_1} \in C_{GU_{n_{f_1}}(q)}(g_{f_1})$ and $\det_{q^2}(c_{f_1})=\delta^k$ as desired.
\end{proof}

Finally, we arrive at our main results.

\begin{theorem} \label{MainReal}
Suppose that $g \in SU_n(q)$ and that $g$ is real as an element of $GU_n(q)$.
\begin{enumerate}
\item[(1)] If $q \not\equiv 3($mod $4)$ or $n$ is odd, then $g$ is real in $SU_n(q)$.
\item[(2)] If $q \equiv 3($mod $4)$ and $n \equiv 2($mod $4)$, then $g$ is real in $SU_n(q)$ if and only if $g$ has an elementary divisor of the form $f(t)^k$ with $k$ odd.
\item[(3)] If $q \equiv 3($mod $4)$ and $n \equiv 0($mod $4)$, then $g$ is real in $SU_n(q)$ if and only if the following holds:

\noindent Write $2^r$ for the smallest power of $2$ such that $2 \leq 2^r < (q^2 - 1)_2$ and $(k)_2 = 2^r$ for some elementary divisor $f(t)^k$ of $g$.  Then $g$ has an even number of elementary divisors of the form $(t \pm 1)^{e}$ such that $(e)_2 = 2^r$ (where the number of these is $0$ if the smallest power $2^r$ dividing the exponent of some elementary divisor is either $1$ or is at least $(q^2 - 1)_2$).
\end{enumerate}
\end{theorem}
\begin{proof} If $q$ is even, then we know $g$ is real in $SU_n(q)$ by Proposition \ref{qeven}, so we assume $q$ is odd.  Replace $g$ by an element in its conjugacy class of the form $\oplus g_{f^*}$, as in Section \ref{initialobs} before Lemma \ref{nounipreal}.  As mentioned there, we have $n = \sum n_{f^*}$, where $n_{f^*}$ is even when $f(t) \neq t \pm 1$.  Therefore, if $n$ is odd, then we must have either $n_{t - 1}$ or $n_{t +1}$ is odd, so that $g$ must have an elementary divisor of the form $(t \pm 1)^k$ with $k$ odd.  By Proposition \ref{oddpower}, it follows that $g$ is real in $SU_n(q)$.

If $q \equiv 1($mod $4)$, we have $g^* = \oplus_{f \neq t \pm 1} g_{f^*}$ is real in $SU_{n^*}(q)$ by Lemma \ref{nounipreal}.  By Lemma \ref{uni1mod4} we have $g_{t-1}$ is real in $SU_{n_{t-1}}(q)$.  Since $-g_{t+1}$ is unipotent, then it is real, and so $g_{t+1}$ is real in $SU_{n_{t+1}}(q)$.  Thus $\oplus g_{f^*}$ is real in $SU_n(q)$, and so $g$ is as well.

Now assume that $q \equiv 3($mod $4)$, and suppose first that $n \equiv 2($mod $4)$.  If $g$ has an elementary divisor of the form $f(t)^k$ with $k$ odd, then $g$ is real in $SU_n(q)$ by Proposition \ref{oddpower}.  Suppose $g$ has no such elementary divisor, and write $g = su$ in terms of its Jordan decomposition.  If $g$ has no elementary divisor of the form $f(t)^k$ with $k$ odd, then it follows from Proposition \ref{Jdecomp} that $u$ has no elementary divisors of the form $(t - 1)^k$ with $k$ odd.  By Corollary \ref{uni2mod43mod4}, it follows that $u$ is not real in $SU_n(q)$.  But if $g$ is then real in $SU_n(q)$ with reversing element $h$, we would have
$$ hgh^{-1} = g^{-1} = s^{-1} u^{-1} = (hsh^{-1})(huh^{-1}), \quad \text{and so} \quad u^{-1} = huh^{-1}$$
by the uniqueness of the Jordan decomposition, a contradiction.  Thus $g$ is not real in $SU_n(q)$.

Finally, assume that $q \equiv 3($mod $4)$ and $4|n$.  Write $g = su$ in its Jordan decomposition.  The main point is that the condition in (3) holds if and only if the conditions in Lemma \ref{uni3mod4} hold, by applying Proposition \ref{Jdecomp}.  Statement (3) then follows immediately from Lemma \ref{keyjordanlemma}.
\end{proof}

\section{Strongly real elements and special orthogonal groups} \label{strongreal}

We finally address the question of strong reality in $SU_n(q)$.  Just as embedded orthogonal subgroups $O^{\pm}_n(q)$ play a key role in the strong reality of $GU_n(q)$, the embedded subgroups $SO^{\pm}_n(q)$ are key in the strong reality of $SU_n(q)$.

\begin{theorem} \label{SUstrong}
Let $g \in SU_n(q)$.  
\begin{enumerate}
\item[(1)] If $n \not\equiv 2($mod $4)$ or $q$ is even, then $g$ is strongly real in $SU_n(q)$ if and only if $g$ is strongly real in $GU_n(q)$.
\item[(2)] If $n \equiv 2($mod $4)$ and $q$ is odd, then $g$ is strongly real in $SU_n(q)$ if and only if $g$ is strongly real in $GU_n(q)$ and $g$ has some elementary divisor of the form $(t \pm 1)^k$ with $k$ odd.
\end{enumerate}
\end{theorem}
\begin{proof} If $g$ is strongly real in $SU_n(q)$, then it is in $GU_n(q)$.  So we assume $g$ is strongly real in $GU_n(q)$ throughout.  Let $h \in GU_n(q)$ be a reversing element of $g$ with $h^2 = 1$, so that $\mathrm{det}(h) = \pm 1$.  If $q$ is even, then necessarily $\mathrm{det}(h) = 1$ and $h \in SU_n(q)$, so $g$ is strongly real in $SU_n(q)$.  So we now assume that $q$ is odd.

By Theorem \ref{gustrong}, if $q$ is odd and $g \in GU_n(q)$ is strongly real, then every elementary divisor of $g$ of the form $(t \pm 1)^{2k}$ appears with even multiplicity, and in particular $g$ is the element of some embedded orthogonal group $O^{\pm}_n(q)$.  If $n$ is odd or divisible by $4$, it follows from Lemma \ref{orthogdet} that $g$ is reversed by an involution from $SO^{\pm}_n(q)$.  In particular, $g$ is reversed by an involution from $SU_n(q)$, and so is strongly real in $SU_n(q)$.  If $n \equiv 2($mod $4)$, and if $g$ has some elementary divisor of the form $(t \pm 1)^k$ with $k$ odd, then it also follows from Lemma \ref{orthogdet} that $g$ is strongly real in $SU_n(q)$.  If $g$ has no such elementary divisor when $n \equiv 2($mod $4)$, then consider $g$ as an element of $SL_n(q^2)$.  It follows from Theorem \ref{slnstrong} that $g$ is not strongly real in $SL_n(q^2)$, so it cannot be strongly real as an element of the subgroup $SU_n(q)$.
\end{proof}

As we mentioned after Lemma \ref{orthogdet}, it was noticed by Gow that if $q$ is odd and $n$ is odd or $4|n$, then every element of $SO_n^{\pm}(q)$ is strongly real.  We can complete the classification of real and strongly real classes of $SO_n^{\pm}(q)$ for $q$ odd as follows.

\begin{theorem} \label{SO2mod4}
Let $q$ be odd, $n \equiv 2($mod $4)$, and $g \in SO^{\pm}_n(q)$.  Then $g$ is real if and only if $g$ is strongly real, and $g$ is real if and only it has an elementary divisor of the form $(t \pm 1)^k$ with $k$ odd.
\end{theorem}
\begin{proof} We first show that $g$ is strongly real in $SO^{\pm}_n(q)$ if and only if $g$ has an elementary divisor of the form $(t \pm 1)^k$ with $k$ odd.  By Lemma \ref{orthogdet}, if $g$ has an elementary divisor of this form, then a reversing involution can be found in $O^{\pm}_n(q)$ with determinant $1$, so $g$ is strongly real in $SO^{\pm}_n(q)$.  Conversely, if $g$ has no elementary divisor of the form $(t \pm 1)^k$ with $k$ odd, then $g$ cannot be strongly real as an element of $SL_n(q)$ by Theorem \ref{slnstrong}(ii), so $g$ cannot be strongly real in the subgroup $SO^{\pm}_n(q)$.

We now show $g$ is real in $SO^{\pm}_n(q)$ if and only if $g$ has an elementary divisor of the form $(t \pm 1)^k$.  If $g$ has such an elementary divisor, then $g$ is strongly real from the above, and so $g$ is real in $SO^{\pm}_n(q)$.  So assume that $g$ has no such elementary divisor.  Since $n \equiv 2($mod $4)$, it follows from Lemma \ref{orthogdet} that there is an inverting involution $h$ of $g$ such that $h \in O^{\pm}_n(q)$ and $\mathrm{det}(h) = -1$.  Any other inverting element of $g$ in $O^{\pm}_n(q)$ must be of the form $ha$ for some $a \in C_{O^{\pm}_n(q)}(g)$.  By \cite[Proposition 16.24]{CaEn04}, a conjugacy class of $O^{\pm}_n(q)$ which is contained in $SO_n^{\pm}(q)$ splits into two conjugacy classes of $SO_n^{\pm}(q)$ if and only if the elements of that class have no elementary divisors of the form $(t \pm 1)^k$ with $k$ odd.  At the same time, by considering corresponding orbits and stabilizers, a class of $O^{\pm}_n(q)$ contained in $SO_n^{\pm}(q)$ splits into two conjugacy classes of $SO_n^{\pm}(q)$ if and only if the centralizer of those elements in $O^{\pm}_n(q)$ are completely contained in $SO_n^{\pm}(q)$ (as mentioned in \cite[Proof of Proposition 16.23]{CaEn04}).  Thus, any element $a \in C_{O^{\pm}_n(q)}(g)$ satisfies $\mathrm{det}(a) = 1$, and so for any reversing element $ha$ of $g$ in $O^{\pm}_n(q)$, we have $\mathrm{det}(ha) = -1$.  Thus $g$ cannot be real in $SO^{\pm}_n(q)$.
\end{proof}

In combining the statements in Theorems \ref{SUstrong} and \ref{SO2mod4}, we obtain the following result on strong reality in $SU_n(q)$ for $q$ odd, which pleasingly resembles strong reality in $GU_n(q)$ for $q$ odd as in the last part of Theorem \ref{gustrong}. 

\begin{corollary} \label{SUSOstrong}
Let $q$ be odd, and let $g \in SU_n(q)$.  Then $g$ is strongly real in $SU_n(q)$ if and only if $g$ is strongly real as an element of some embedded $SO^{\pm}_n(q)$.
\end{corollary}

\end{document}